\documentclass[12pt]{article}                   
\usepackage[square,sort,comma,numbers]{natbib}
\usepackage{graphicx}
\usepackage{enumitem}

\usepackage{authblk}
\usepackage[]{graphicx}
\usepackage{amssymb}
\usepackage{mathtools}
\usepackage{mathrsfs}
\usepackage{bm}
\usepackage{bbm}
\usepackage{caption}
\usepackage{subcaption}
\usepackage{tabu}
\usepackage{longtable}
\usepackage{float}
\usepackage[ruled,vlined]{algorithm2e}
\usepackage{geometry}
\usepackage[sort]{natbib}
\usepackage[scientific-notation=true]{siunitx}
\usepackage{moresize}
\usepackage{appendix}
\usepackage{enumitem}

 \newcommand{\mydist}{\pi}
\newcommand{\R}{\mathbb R}
\newcommand{\beq}{\begin{equation}}
\newcommand{\eeq}{\end{equation}}
\newcommand{\delim}{\; : \;}
\newcommand{\N}{\mathbb N}
\newcommand{\Var}{\mathrm{Var}}
\newcommand{\Ind}{\mathbbm{1}}

\newcommand{\Em}{\mathbb E}
\newcommand{\Pm}{\mathbb P}

\newcommand{\e}{\text{e}}

\newcommand{\bX}{\mathbf{X}}
\newcommand{\bx}{\mathbf{x}}

\newcommand{\bZ}{\mathbf{Z}}

\newcommand{\bY}{\mathbf{Y}}

\newcommand{\bw}{\mathbf{w}}

\newcommand{\scX}{\mathscr{X}}

\newcommand{\scZ}{\mathscr{Z}}

\newcommand{\di}{\text{d}}

\def\qed{\halmos}
\newcommand{\halmos}{\vspace{3mm} \hfill \mbox{$\Box$}}

\newcommand{\Ct}[1]{\widehat{Z}_{#1}}
\newcommand{\Ht}[1]{\widehat{\Phi}_{#1}}
\newcommand{\Pt}[1]{\widehat{P}_{#1}}
\newcommand{\Rt}[1]{\widehat{R}_{#1}}
\newcommand{\Hfunc}{\varphi}

\newcommand{\fmc}[1]{\widehat{\mu}_{#1}}
\newcommand{\pmc}[1]{\widehat{\nu}_{#1}}

\newcommand{\mypi}[1]{{g}_{#1}}

\newcommand{\wmin}{\underline{w}}
\newcommand{\wmax}{\overline{w}}

\newcommand{\bset}{\scX_b}
\newcommand{\bsetcard}{\left| \bset \right|}
\newcommand{\bwset}{\scX_{b-w_1}}
\newcommand{\bwsetcard}{\left| \bwset \right|}

\newcommand{\bwsubset}{\bset^{(w_1)}}
\newcommand{\bsubset}{\bset^{(-w_1)}}
\newcommand{\bminuswsubset}{\bwset^{(w_1)}}
\newcommand{\bminussubset}{\bwset^{(-w_1)}}

\newcommand{\perf}{S}
\newcommand{\treshold}{\gamma}
\newtheorem{remark}{Remark}
\newtheorem{theorem}{Theorem}
\newtheorem{proof}{Proof}
\newtheorem{definition}{Definition}
\newtheorem{proposition}{Proposition}
\newtheorem{lemma}{Lemma}

\begin{document}

\title{Stratified Splitting for Efficient Monte Carlo\\ Integration}

\author[]{Radislav Vaisman}
\author[]{Robert Salomone}
\author[]{Dirk P. Kroese}

\affil[]{\small School of Mathematics and Physics, The University of Queensland \\ 
 Brisbane, Australia \\ \vspace{1em} \footnotesize
 r.vaisman@uq.edu.au \\  robert.salomone@uqconnect.edu.au \\ kroese@maths.uq.edu.au}


\date{}




\maketitle
\begin{abstract}
  The efficient evaluation of high-dimensional integrals is of importance in both theoretical and practical fields of science, such as data science, statistical physics,  and machine learning.
  However, due to the curse of dimensionality,
  deterministic numerical methods are inefficient in high-dimensional
  settings. Consequentially, for many practical problems one must
  resort to Monte Carlo estimation. In this paper, we introduce a
  novel Sequential Monte Carlo technique called Stratified
  Splitting. The method provides unbiased estimates and can handle
  various integrand types including indicator functions, which are
  used in rare-event probability estimation problems. Moreover,
  we demonstrate that a variant of the algorithm can achieve polynomial
  complexity. 
The results of our numerical experiments suggest that
  the Stratified Splitting method is capable of delivering accurate
  results for a variety of integration problems. 
\end{abstract}

\section{Introduction}
We  consider the evaluation of expectations and integrals of the form
\beq\label{eq.integral}
	\Em_f \left[ \Hfunc(\bX) \right]=  \sum_{\scX} \Hfunc({\bf x}) f({\bf x}) \quad \text{or} \quad \Em_f \left[ \Hfunc(\bX) \right]=\int_{\scX} \Hfunc({\bf x}) f({\bf x})\, \di {\bf x}, \nonumber
\eeq
where $\bX \sim f$ is a random variable   taking values in a set ${\scX} \subseteq \R^d$,  $f$ is a probability density function (pdf) with respect to the Lebesgue or counting measure, and  $\Hfunc:\scX \rightarrow \R$ is a real-valued function. 

The evaluation of such high-dimensional integrals is of critical importance in many scientific areas, including  statistical inference \citep{Gelman}, rare-event probability estimation \citep{asmussen2007stochastic}, machine learning \citep{Russell2009,Koller2009PGM1795555}, and cryptography \citep{mcgrayne2011theory}. An important application is the  calculation of the normalizing constant of a probability distribution, such as the marginal likelihood (model evidence).  However, often obtaining even a reasonably accurate estimate of $\Em_f \left[ \Hfunc(\bX) \right]$ can be hard \citep{rubinstein17}.  
Deterministic computation methods that use
Fubini's theorem \citep{friedman1980} and quadrature rules or
extrapolations \citep{opac-b1079222} suffer from the curse of
dimensionality, with the number of required function evaluations
growing exponentially with the dimension. Many other deterministic and
randomized methods have been 
proposed to estimate high-dimensional integrals. Examples include Bayesian
quadrature, sparse grids, and various Monte Carlo, quasi-Monte Carlo,
Nested Sampling, 
and Markov Chain Monte Carlo (MCMC) algorithms
\citep{OHAGAN1991245,Morokoff1995218,
  newman1999monte,Heiss200862,skilling2006, Kroese2011}. 
There are also procedures
based on the sequential Monte Carlo (SMC) approach \citep{RSSB:RSSB553} that provide
consistent and unbiased estimators that have asymptotic
normality. However, many popular
methods are
not unbiased and some are not known to be consistent. For example,
Nested Sampling produces biased estimates and even its consistency, 
when Markov Chain Monte Carlo (MCMC) is used, remains an open problem
~\citep{ChopinRobert2010}.  
An alternative criterion that one might use to describe the efficiency
of algorithms is {\em computational complexity}, which often used in a computer
science context. Here, an algorithm is considered to be efficient if
it scales polynomially (rather than
exponentially) in the size of the problem.  

In this paper, we propose a novel SMC
approach for reliable and fast estimation of high-dimensional
integrals.  Our method extends the Generalized Splitting (GS)
algorithm of \cite{Botev2012}, to allow the estimation of quite
general integrals.  In addition, our algorithm is specifically
designed to perform efficient sampling in regions of ${\scX}$ where
$f$ takes small values and $\Hfunc$ takes large values.
 In particular, we present a way of implementing stratification for
 variance reduction in the absence of knowing the strata
 probabilities. A benefit of the proposed  Stratified Splitting
 algorithm (SSA) is that it provides an  unbiased  estimator of $\Em_f
 \left[ \Hfunc(\bX) \right]$, and that it can be analyzed in a
 non-asymptotic setting. In particular, we prove that a simplified 
 version of SSA can provide polynomial complexity, under certain
 conditions.
We give a specific example of a $\#$P complete problem where
polynomial efficiency can be achieved, providing a polynomial bound on
the number of samples required to
achieve a predefined error bound. 

The SSA uses a  stratified sampling scheme (see,
e.g., \cite{rubinstein17}, Chapter 5), defining a partition of
the state space into strata, and using the law of total probability to
deliver an estimator of the value of the integral.  To do so, one
needs to obtain a sample population from each strata and know the
exact probability of each such strata. Under the classical stratified
sampling framework, it is assumed that the former is easy to achieve
and the latter is known in advance. However, such favorable scenarios
are rarely seen in practice. In particular, obtaining samples from
within a stratum and estimating the associated probability that a
sample will be within this stratum is hard in general
\citep{JerrumValiantVazirani}. To resolve this issue, the SSA
incorporates a multi-level splitting mechanism
\citep{Kahn1951,Botev2012,Rubinstein2010,Duan2017} and uses an
appropriate MCMC method to sample from conditional densities
associated with a particular~stratum.

The rest of the paper is organized as follows. In Section
\ref{sec.ssa} we introduce the SSA, explain its correspondence to a
generic multi-level sampling framework, and prove that the SSA
delivers an unbiased estimator of the expectation of interest.  In
Section \ref{sec.analysis}, we provide a rigorous complexity analysis of the
approximation error of the proposed method, under a simplified
setting.
  In Section
\ref{sec.exact.example}, we introduce a challenging estimation problem
called the weighted component model, and demonstrate that SSA can
provide an arbitrary level of precision for this problem with
polynomial complexity in the size of the problem.  In Section
\ref{sec.numerics}, we report our numerical findings on various test
cases that typify classes of problems for which the SSA is of
practical interest. Finally, in Section \ref{sec.conclusion} we
summarize the results and discuss possible directions for future
research. Detailed proofs are given in the appendix.

\section{Stratified splitting algorithm}\label{sec.ssa}

\subsection{Generic multilevel splitting framework}

We begin by considering a very generic multilevel splitting  framework, similar to \citep{RSSB:RSSB280}. Let $\bX \sim f$ be a random variable   taking values in a set ${\scX}$. Consider a decreasing sequence of sets  ${\scX} = {\scX}_0  \supseteq \cdots  \supseteq {\scX}_{n}=\emptyset$ and define
${\scZ}_t  = {\scX}_{t-1} \setminus {\scX}_{t}$, for $t=1,\ldots,n$. Note that ${\scX}_{t-1} =  \bigcup_{i=t}^{n}{{\scZ}_i}$ and that $\{ {\scZ}_t \}$  yields a partition of ${\scX}$; that is,
\beq\label{eq.partition}
	{\scX} = \bigcup_{t=1}^{n} {\scZ}_t, \quad {\scZ}_{t_1} \cap {\scZ}_{t_2} = \emptyset \quad \mathrm{for} \quad 1 \leq t_1 <t_2 \leq n.
\eeq
We can define a sequence of conditional pdfs
\begin{align}\label{eq.conditional1}
	f_t(\bx) &= f(\bx \mid \bx \in {\scX_{t-1}})
	 =  \frac{f(\bx) \Ind{\{ \bx \in \scX_{t-1} \}}}{\Pm_f(\bX \in \scX_{t-1})}   \quad \text{for} \quad t=1,\ldots,n,
\end{align}
where $\Ind$ denotes the indicator function. Also, define
\beq\label{eq.conditional2}
	\mypi{t}(\bx) = f\left(\bx \mid \bx \in {\scZ}_t \right)   =  \frac{f(\bx) \Ind{\{ \bx \in \scZ_t \}}}{\Pm_f(\bX \in \scZ_t)}   \quad \text{for} \quad t=1,\ldots,n.
\eeq

Our main objective is  to  sample from the pdfs $f_t$ and $g_t$ in \eqref{eq.conditional1} and \eqref{eq.conditional2}, respectively. To do so, we first formulate a  generic multilevel splitting framework,  given in Algorithm \ref{alg.gen.split}.

\begin{algorithm}[h]
{
\SetAlgoSkip{}
\DontPrintSemicolon
\SetKwInOut{Input}{input}\SetKwInOut{Output}{output}
\Input{${\scX}_0,\ldots,{\scX}_{n}$ and $\{f_t, \mypi{t}\}_{1 \leq t \leq n}$.} 
\Output{Samples from $f_t$ and $\mypi{t}$ for $1 \leq t \leq n$.}
Create a multi-set ${\cal X}_1$ of samples from $f_1$.\\
\For{$t = 1$ \KwTo $n$}
	{		
		Set ${\cal Z}_t \leftarrow {\cal X}_{t} \cap \scZ_t$.  \\
		Set ${\cal Y}_{t} \leftarrow {\cal X}_{t} \setminus {\cal Z}_t$.\\
		\If{$t < n$}{
		
		Create a multi-set  ${\cal X}_{t+1}$ of samples (particles) from $f_{t+1}$, (possibly) using elements of the ${\cal Y}_{t}$ set. \tcc{This step is called the splitting or the rejuvenation step.}
		}
	}
\KwRet multi-sets  $\{{\cal X}_t\}_{1 \leq t \leq n}$, and $\{{\cal Z}_t\}_{1 \leq t \leq n}$.
\caption{\footnotesize{Generic multilevel splitting framework} \label{alg.gen.split}}
}
\end{algorithm}


Note that the samples in $\{{\cal X}_t\}_{1 \leq t \leq n}$ and $\{{\cal Z}_t\}_{1 \leq t \leq n}$ are distributed according to $f_{t}$ and $\mypi{t}$, respectively, and these samples can be used to handle several tasks. In particular, the $\{{\cal X}_t\}_{1 \leq t \leq n}$ sets allow one to handle the general non-linear Bayesian filtering problem \citep{BootstrapFilter}. Moreover, by tracking the  cardinalities of   the sets $\{{\cal X}_t\}_{1 \leq t \leq n}$ and $\{{\cal Z}_t\}_{1 \leq t \leq n}$, one is able to tackle  hard rare-event probability estimation problems, such as delivering  estimates of $\Pm_f\left(\bX \in \scZ_{n} \right)$ \citep{Botev2012,Kroese2011,Vaisman1}. 
Finally, it was recently shown by \cite{Vaisman20161} that Algorithm \ref{alg.gen.split} can be used as a powerful variance minimization technique for any general SMC procedure.
In light of the above, we propose taking further advantage of the sets $\{{\cal X}_t\}_{1 \leq t \leq n}$ and $\{{\cal Z}_t\}_{1 \leq t \leq n}$, to obtain an estimation method suitable for general  integration problems.

\subsection{The SSA set-up}\label{sec.ssa.setup}

Following the above multilevel splitting framework, it is convenient
to construct the sequence of sets $\{{\scX}_t\}_{0 \leq t \leq n}$ by
using a performance function $S: {\scX} \rightarrow \R$, in such a way
that $\{{\scX}_t\}_{0 \leq t \leq n}$ can be written as {\em super}
level-sets of $S$ for chosen levels $\gamma_0,\ldots,\gamma_n$, where
$\gamma_0$ and $\gamma_n$ are equal to $\inf_{\bx \in {\scX}}S(\bx)$
and $\sup_{\bx \in {\scX}} S(\bx)$, respectively.
In particular, ${{\scX}_t = \left\{ \bx \in {\scX} : S(\bx) \geq
\gamma_{t} \right\}}$ for $t=0,\ldots,n$. The partition
$\{{\scZ}_t\}_{1 \leq t \leq n}$, and the densities $\{f_t\}_{1 \leq t
  \leq n}$ and $\{\mypi{t}\}_{1 \leq t \leq n}$, are defined as before
via \eqref{eq.partition}, \eqref{eq.conditional1}, and
\eqref{eq.conditional2}, respectively. Similarly, one can define a
sequence of {\em sub} level-sets of $S$; in this paper we use the
latter for some cases and whenever appropriate.


Letting $z_t \stackrel{\tiny\mathrm{def}}{=} \Em_f\left[\Hfunc\left(\bX\right) \mid \bX \in  {\scZ}_t \right] \Pm_f \left(\bX \in  {\scZ}_t \right)$  for~$t=1,\ldots, n$, and combining \eqref{eq.partition}  with the law of total probability, we arrive at
\beq \label{eq.cond.exp.analitical}
	z \stackrel{\tiny\mathrm{def}}{=} \Em_f\left[\Hfunc\left(\bX\right)\right] = \sum_{t=1}^{n} { \Em_f\left[\Hfunc \left(\bX\right) \mid \bX \in  {\scZ}_t \right] \Pm_f\left(\bX \in  {\scZ}_t \right)} = \sum_{t=1}^{n}{z_t}.
\eeq

The  SSA proceeds with the  construction of  estimators $\Ct{t}$  for $z_t$ for  $t=1,\ldots, n$ and,  as soon as these are available, we can use \eqref{eq.cond.exp.analitical} to deliver the SSA  estimator for $z$, namely
$\Ct{} = \sum_{t=1}^{n} {\Ct{t}}$.

For $1 \leq t \leq n$, let $\varphi_t\stackrel{\tiny\mathrm{def}}{=}\Em_f\left[\Hfunc\left(\bX\right) \mid \bX \in  {\scZ}_t \right]$,
$p_t\stackrel{\tiny\mathrm{def}}{=}\Pm_f\left(\bX \in  {\scZ}_t \right)$, and let $\Ht{t}$ and $\Pt{t}$ be estimators of $\varphi_t$ and $p_t$, respectively. We define $\Ct{t} = \Ht{t}\,\Pt{t}$, and  recall that, under the multilevel splitting framework, we obtain the  sets $\{{\cal X}_t\}_{1 \leq t \leq n}$,  and  $\{{\cal Z}_t\}_{1 \leq t \leq n}$. These sets are sufficient to obtain unbiased estimators $\{\Ht{t}\}_{1 \leq t \leq n}$ and $\{\Pt{t}\}_{1 \leq t \leq n}$, in the following way.

\begin{enumerate}
	\item We define $\Ht{t}$ to be the (unbiased) Crude Monte Carlo (CMC) estimator of $\Hfunc_t$, that is,
	\[
		\Ht{t} = \frac{1}{|{\cal Z}_t|} \sum_{\bZ \in {\cal Z}_t} { \Hfunc(\bZ) } \quad \text{for all} \;\; t=1,\ldots,n.
	\]
	
	\item\label{rt.explain} The estimator $\Pt{t}$ is defined  similar to the one used in the  Generalized Splitting (GS) algorithm of \cite{Botev2012}. In particular, the GS product estimator is defined as follows.  Define the level entrance probabilities $r_0 \stackrel{\tiny\mathrm{def}}{=} 1$, $r_t \stackrel{\tiny\mathrm{def}}{=}  \Pm_f \left(\bX \in \scX_{t} \mid \bX \in \scX_{t-1} \right)$ for $t=1,\ldots,n$, and note that $\Pm_f \left(\bX \in \scX_{t} \right) = \prod_{i=0}^{t}{r_i}$. Then, for $t= 1,\ldots, n$, it holds that
	\begin{align*}
		p_t &= \Pm_f\left(\bX \in \scZ_t \right) = \Pm_f \left(\bX \in \scX_{t-1} \right) - \Pm_f \left(\bX \in \scX_{t} \right) \\
		& = \prod_{i=0}^{t-1}{r_i} - \prod_{i=0}^{t}{r_i}
		 = (1-r_{t})\,\prod_{i=0}^{t-1}{r_i}.
	\end{align*}
This suggests the estimator
$\Pt{t} = (1-\Rt{t})\prod_{i=0}^{t-1}\Rt{i}$
 for $p_t$, where $\Rt{0}\stackrel{\tiny\mathrm{def}}{=} 1$ and  $\Rt{t} = \frac{|{\cal Y}_{t}|}{|{\cal X}_{t}|} = \frac{|{\cal X}_{t} \setminus {\cal Z}_t|}{|{\cal X}_{t}|}$ for all $t=1,\ldots,n$.
\end{enumerate}
In practice, obtaining the $\{{\cal X}_t\}_{1 \leq t \leq n}$ and $\{{\cal Z}_t\}_{1 \leq t \leq n}$ sets requires
 sampling  from the conditional
pdfs  in \eqref{eq.conditional1} and \eqref{eq.conditional2}. However, for many specific
applications, designing such a procedure can be extremely
challenging. Nevertheless,  we can use the
$\mathcal{Y}_{t}={\cal X}_{t} \setminus {\cal Z}_t$ set from iteration
$t$ to sample $\mathcal{X}_{t+1}$ from $f_{t+1}$ for each
$t=1,\ldots,n-1$, via MCMC. In particular, the particles from the
$\mathcal{Y}_{t}$ set can be ``split'', in order to construct the
desired set $\mathcal{X}_{t+1}$  for the next iteration, using a
Markov transition kernel $\kappa_{t+1}\left(\cdot \mid \cdot  \right)$
whose stationary pdf is  $f_{t+1}$, for each~$t=1,\ldots,n-1$.
Algorithm \ref{alg.SSA} summarizes the general procedure for the~SSA. 
 The particular splitting step described in Algorithm \ref{alg.SSA} is
 a popular choice \cite{Botev2012, Botev2008, Rubinstein2009},
 especially for hard problems with unknown convergence behavior of the
 corresponding Markov chain. 

\begin{algorithm}[H]
  {\footnotesize
  \SetAlgoSkip{}
  \DontPrintSemicolon
  \SetKwInOut{Input}{input}\SetKwInOut{Output}{output}
  \Input{A set ${\scX}$, a pdf $f$, the functions $\Hfunc \delim {\scX} \rightarrow \R$ and $S \delim {\scX} \rightarrow \R$, a  sequence of levels $\gamma_0,\ldots,\gamma_{n}$,  and the sample size $N \in \N$.}
  \Output{$\Ct{}$ --- an  estimator of $z = \Em_f \left[ \Hfunc(\bX) \right]$.}
  Set $\Rt{0} \leftarrow 1$,  ${\cal X}_1 \leftarrow \emptyset$, and $f_1 \leftarrow f$.\\
  \For{$i = 1$ \KwTo $N$}{draw $\bX \sim f_1(\bx)  $ and  add $\bX$ to ${\cal X}_1$.}	
  \For{$t=1$ \KwTo $n$}
  {
   Set $\mathcal{Z}_t \leftarrow \left\{ \mathbf{X} \in  \mathcal{X}_{t} \;:\;  \mathbf{X} \in {\scZ}_t  \right\}$ and  $\mathcal{Y}_{t} \leftarrow {\scX}_{t} \setminus \mathcal{Z}_t$. \\
   Set $\Ht{t} \leftarrow \frac{1}{|\mathcal{Z}_t|} \sum_{\mathbf{X} \in  \mathcal{Z}_t}{\Hfunc(\mathbf{X})}$.\\
   Set $\Rt{t} \leftarrow  \frac{|\mathcal{Y}_{t}|}{N} $, and
    $\Pt{t}\leftarrow\left(1-\Rt{t}\right)\prod_{j=0}^{t-1}{\Rt{j}}$.\\
    Set $\Ct{t} \leftarrow \Ht{t} \, \Pt{t}$.\\
  \If{$t < n$}{
  \tcc{Performing splitting to obtain ${\cal X}_{t+1}$. }
  Set $\mathcal{X}_{t+1} \leftarrow \emptyset$ and draw  $K_{i}  \sim   \mathrm{\sf Bernoulli}(0.5)$, for $i=1,\ldots, \left| \mathcal{Y}_{t} \right|$, such that  $\sum_{i=1}^{ \left| \mathcal{Y}_{t} \right|} K_{i} = N \; \mathrm{mod} \; \left| \mathcal{Y}_{t} \right|$. \\
  \For{$\bY \in  \mathcal{Y}_{t}$}{
    
     Set $M_i  \leftarrow  \left\lfloor \frac{N}{\left| \mathcal{Y}_{t} \right|}  \right \rfloor + K_{i}$ and $\bX_{i,0} \leftarrow \bY$.\\
    
    \For{$j=1$ \KwTo $M_{i}$}{		
         Draw $\bX_{i,j} \sim \kappa^{\tau}_{t+1}\left(\cdot \mid \bX_{i,j-1} \right)$ (where $\kappa^{\tau}_{t+1}(\cdot \mid \cdot)$ is a $\tau$-step transition kernel using $\kappa_{t+1}(\cdot \mid \cdot)$), and add $\bX_{i,j}$ to $\mathcal{X}_{t+1}$.
        }
    }
  }
  }
  \KwRet  $\Ct{} = \sum_{t=1}^n{\Ct{t}}$.
  
  \caption{\footnotesize{SSA for estimating $z=\Em_f\left[\Hfunc\left(\bX\right)\right]$} \label{alg.SSA}}
  }
  \end{algorithm}

\begin{remark}[Independent SSA]\label{rem.split.strategy} \normalfont
  Obviously, the SSA  above produces dependent samples in sets ${\cal X}_t$
  and ${\cal Z}_t$.
It will be convenient to also consider a version of SSA where the
samples are {\em independent}. This can for example be achieved by
simulating multiple independent runs of SSA, until $N$ realizations
$\bX \in {\cal X}_t$ have been produced.
Of course this {\em Independent SSA} (ISSA) variant is much less efficient than the original algorithm,
but the point is that the computational effort (the run time complexity) 
 remains polynomial --- $\mathcal{O}(C^2 n)$ versus $\mathcal{O}(C n)$, where $C$ is the
 complexity of SSA for each level. 
Hence, from a {\em complexity} point of view we may as well consider the
independent SSA variant. This is what will be done in Theorem~\ref{thm.mcmc}.

  \end{remark}
  \begin{theorem}[Unbiased estimator]\label{thm.unbiased}  Algorithm \ref{alg.SSA} outputs an unbiased estimator; that is, it holds that
  $\Em\left[\Ct{}\right] = \Em_f \left[\Hfunc \left(\bX \right) \right]=z.$
  \end{theorem}
  \begin{proof}
  See Appendix \ref{appendix.proofs}.
  \end{proof}
%
We next proceed with a clarification for a few remaining practical
issues regarding the SSA.

\smallskip

\emph{Determining the SSA levels.} 
It is often difficult to make an educated guess  how to set the values of the level thresholds. However, the SSA requires the values of $\{\gamma_t \}_{1 \leq t \leq n }$ to be known in advance, in order to ensure that the estimator is unbiased. To resolve this issue, we perform a single pilot run of Algorithm \ref{alg.SSA} using a  so-called rarity parameter $0 < \rho < 1$. In particular, given  samples from an ${\cal X}_t$ set, we  take the $\rho \, |{\cal X}_t|$ performance quantile as the value of the corresponding level $\gamma_t$, and  form the next level set. Such a pilot run helps  to establish a set of threshold values adapted to the specific problem. After the completion of the pilot run we simply continue with a regular execution of Algorithm~\ref{alg.SSA} using the level threshold values observed in the pilot run.
\smallskip

\emph{Controlling the SSA error.} A common practice when working with a Monte Carlo
algorithm that outputs an unbiased estimator, is to run it for $R$ independent replications to obtain $\Ct{}^{(1)},\ldots,\Ct{}^{(R)}$, and report
the average value. Thus, for a final estimator, we take
\[
	\Ct{} = {R^{-1}}\sum_{j=1}^R{\Ct{}^{(j)}}.
\]
 To measure the quality of the SSA output, we use the estimator's relative error (RE), which is defined by
\[
\mathrm{RE} = \frac{\sqrt{\Var\left(\Ct{} \right)}} {\Em\left[ \Ct{}\right] \sqrt{R}}.
\]

As the variance and expectation of the estimator are not known explicitly, we report an estimate of the relative error by estimating both terms from the result of the $R$ runs.

\smallskip

\section{Complexity}\label{sec.analysis}
In this section, we present a theoretical complexity analysis of the
Independent SSA variant, discussed
in Remark~\ref{rem.split.strategy}.  Our focus is on establishing
time-complexity results, and thus our style of analysis is similar to
that used for approximate counting algorithms. For an extensive
overview, we refer to \cite[Chapter 10]{Mitzenmacher}.  We begin with
a definition of a randomized algorithm's efficiency.
\begin{definition}[\cite{Mitzenmacher}]\label{def.approx}
A~randomized algorithm gives an $(\varepsilon, \delta)$-approximation for the value $z$
if the output $\Ct{}$ of the algorithm satisfies
\[
	 \Pm \left(z (1-\varepsilon) \leq \Ct{} \leq z(1+ \varepsilon) \right) \geq 1-\delta.
\]
\end{definition}
With the above definition in mind, we now aim  to specify the sufficient conditions for the SSA to provide an
 $(\varepsilon, \delta)$-approximation to $z$.  A key component in our  analysis is to construct a Markov chain $\left\{ X^{(m)}_{t}, \, m \geq 0 \right\}$ with  stationary pdf $f_t$ (defined in \eqref{eq.conditional1}),  for all $1 \leq t \leq n$, and to consider the speed of convergence of the distribution of $ X^{(m)}_{t}$ as $m$ increases. Let $\mu_t$ be the probability distribution corresponding to $f_t$, so
\[
	\mu_t(A) = \int_A{ f_t(u)\lambda(\mathrm{d}u) },
\]
for all Borel sets $A$,
where $\lambda$ is some base measure, such as the Lebesgue or counting measure.
To proceed, we have 
\[
	\kappa_t^\tau(A \mid \bx) = \Pm\left(  X^{(\tau)}_{t} \in A \mid X^{(0)}_{t} = \bx \right),
\]
for the $\tau$-step transition law of the Markov chain. Consider the { \em total variation distance} between  $\kappa_t^\tau(\cdot \mid \bx)$ and   $\mu_t$,  defined as:
\[
 \| \kappa_t^\tau(\cdot \mid \bx) - \mu_t \|_{\mathrm{TV}} = \sup_{A}{ \left|  \kappa_t^\tau(A \mid \bx) - \mu_t(A) \right|}.
\]
An essential ingredient of our analysis is the so-called mixing time (see \cite{roberts2004general} and \cite{opac-b1128575} for an extensive overview), which is defined  as 
\[
	\tau_{\mathrm{mix}}(\varepsilon,\bx) = \min{ \{ \tau \, : \, \| \kappa_t^\tau(\cdot \mid \bx) - \mu_t \|_{\mathrm{TV}} \leq \varepsilon \}}.
\]
Let $\fmc{t}  = \kappa_t^\tau(\cdot \mid \bx \,) $  be the SSA sampling distribution at steps $1 \leq t \leq n$  where, for simplicity, we suppress $\bx$ in the notation of $\fmc{t}$.

Finally, similar to $\mu_t$ and $\fmc{t} $, let $\nu_t$  be the
probability distribution corresponding to the pdf $\mypi{t}$ (defined
in \eqref{eq.conditional2}), and let  $\pmc{t}$ be the  SSA sampling
distribution, for all $1 \leq t \leq n$.   

Theorem \ref{thm.mcmc}
details the main efficiency result for the ISSA. We reiterate that from a
{\em complexity} point of view 
the independence setting does not impose a
theoretical limitation, as the run time complexity remains polynomial. An advantage
is that by using ISSA  we can engage powerful concentration
inequalities \citep{chernoff1952,hoefding}.  

\begin{theorem}[Complexity of the ISSA]\label{thm.mcmc}
  Let $ \Hfunc$ be a strictly positive real-valued function, $a_t =
  \min_{\bx \in \scZ_t}{\{ \Hfunc(\bx) \}}$, $b_t = \max_{\bx \in
    \scZ_t}{\{ \Hfunc(\bx) \}}$, and $\underline{r}_t =
  \min{\{r_t,1-r_t\}}$ for $1 \leq t \leq n$.
Let $\fmc{t}$ and $\pmc{t}$ be the ISSA sampling distributions at steps $1 \leq t \leq
n$, for ${\cal X}_t$ and ${\cal Z}_t, respectively$.  
  Then, the ISSA  gives an  $(\varepsilon, \delta)$-approximation to $z = \Em_f[\Hfunc(\bX)]$, provided that for all $1 \leq t \leq n$ the following holds.
  \begin{enumerate}
    
    \item \quad \label{thm.mcmc.condition1} $\displaystyle
 \| \fmc{t} - \mu_t \|_{\mathrm{TV}} \leq \frac{\varepsilon \,\underline{r}_t}{32n} \quad \text{and} \quad 		|{\cal X}_t| \geq \frac{3072  \,n^2 \ln (4n^2 /\delta)}{\varepsilon^2 \underline{r}_t^2}.$	
    
    \item \quad 
      $\| \pmc{t} - \nu_t \|_{\mathrm{TV}} \leq \frac{\varepsilon\, a_t}{16(b_t - a_t)} \quad \text{and} \quad |{\cal Z}_t| \geq \frac{128(b_t-a_t)^2\ln (4n/\delta)}{\varepsilon^2 a_t^2}$.	
  \end{enumerate}
  \end{theorem}
  
  \begin{proof} See Appendix \ref{appendix.proofs}.
  \end{proof}
  


  In some cases, the distributions of the states in $\scX_t$ and $\scZ_t$ generated by  Markov chain defined by the kernel $\kappa_t^{\tau}$, approach the target distributions $\mu_t$ and $\nu_t$ very fast. This occurs for example when there exists a polynomial in $n$ (denoted by $\mathcal{P}(n)$), such that the mixing time \citep{opac-b1128575} is bounded by $\mathcal{O}(\mathcal{P}(n))$,
  $(b_t-a_t)^2/a^2 = \mathcal{O}(\mathcal{P}(n))$, and  $\underline{r}_t = \mathcal{O}(1/\mathcal{P}(n)) $  for all $1 \leq t \leq n$.
  In this case, the ISSA becomes a {\it fully polynomial randomized approximation scheme} (FPRAS) \citep{Mitzenmacher}. In particular, the ISSA results in a desired $(\varepsilon, \delta)$-approximation to $z = \Em_f[\Hfunc(\bX)]$ with running time bounded by a polynomial in $n$, $\varepsilon^{-1}$, and $\ln(\delta^{-1})$. Finally, it is important to note that an FPRAS algorithm for such problems is essentially the best result one can hope to achieve \citep{Jerrum96}.
  
  \smallskip
  
  We next illustrate the use of 
  Theorem \ref{thm.mcmc} with an example of a difficult problem for which the ISSA provides an FPRAS.

  \section{FPRAS for the weighted component model}\label{sec.exact.example}
  Consider a system of $k$ components. Each  component $i$  generates a specific amount of benefit, which is given by a positive real number $w_i$, $i=1,\ldots,k$. In addition, each  component can be operational or not. 
  
  Let $\bw=(w_1,\ldots,w_k)^\top$ be the column vector of component weights (benefits), and $\bx = (x_1,\ldots, x_k)^\top$ be a binary column vector, where  $x_i$ indicates the $i$th component's operational status for $1 \leq i \leq k$. That is, if the component $i$ is operational $x_i = 1$, and $x_i = 0$ if it is not. Under this setting, we define the system  performance as
  \[
    \perf(\bw,\bx)= \sum_{i=1}^k w_i \, x_i = \bw^\top \bx.
  \]
  We further assume that all elements are independent of each other, and that each element is operational with probability $1/2$ at any given time. For the above system definition, we might be interested in the following questions.
  
  \begin{enumerate}
      
  \item\label{exect.prob2} \emph{Conditional expectation estimation.}  Given a minimal threshold performance  $\treshold \leq \sum_{i=1}^k w_i$, what is the expected system performance? That is to say, we are interested in the calculation of
  \beq\label{eq.sim.prob1}
  \Em \left[\perf(\bw,\bX) \mid \perf(\bw,\bX) \leq \treshold \right],
  \eeq
  where $\bX$ is a $k$-dimensional binary vector generated uniformly at random from the $\{0,1 \}^k$ set.
  This setting  appears (in a more general form), in a  portfolio credit risk analysis \citep{Glasserman2}, and will be discussed in Section \ref{sec.numerics}.
  
  \item\label{exect.prob1} \emph{Tail probability estimation \citep{asmussen2007stochastic}.} Given the minimal threshold performance  $\treshold$, what is the probability that the overall system performance is smaller than $\treshold$? In other words, we are interested in calculating
    \beq\label{eq.sim.prob2}
      \Pm(\perf(\bX) \leq \treshold) =  \Em\left[\Ind{\{ \perf(\bX) \leq \treshold\}} \right].
    \eeq
  
  \end{enumerate}
  
  The above problems are both difficult, since a  uniform  generation of $\bX \in \{0,1 \}^k$, such that
  $\bw^\top\bX \leq \treshold$,  corresponds to the 
  {\em knapsack} problem, which belongs to \#P complexity class ~\citep{Valiant79,SinclairCube}. In this section, we show how one can construct an FPRAS for both problems under the mild condition that the difference between the minimal and the maximal weight in the $\bw$ vector is not large. This section's main  result is summarized next.
  \begin{proposition}	\label{prop.ssa.fpras}
  Given a weighted component model with $k$ weights,  $\bw = (w_1,\ldots,w_k)$ and a threshold $\treshold$, let $\wmin = \min\{ \bw \}$, $\wmax = \max\{ \bw \}$. Then,  provided that $\wmax = \mathcal{O}\left(\mathcal{P}(k)\right)\wmin$, there exists an FPRAS for the estimation of both \eqref{eq.sim.prob1} and \eqref{eq.sim.prob2}.
  \end{proposition}
  Prior to stating the proof of Proposition \ref{prop.ssa.fpras}, define
  \beq\label{eq.x.exact}
    \scX_b = \left\{\bx \in \{ 0,1\}^k \; : \; \sum_{i=1}^k w_i \, x_i \leq b \right\} \quad \text{for  }b \in \R,
  \eeq
  and let $\mu_b$ be the uniform distribution on the $\scX_b$ set.  \cite{SinclairCube} introduce an MCMC algorithm that is capable of sampling from the $\scX_b$ set almost uniformly at random. In particular, this algorithm can sample  $\bX \sim \fmc{b}$, such that $\left\|\fmc{b} - \mu_b \right\|_{\mathrm{TV}} \leq \varepsilon$. Moreover, the authors show that their Markov chain mixes rapidly, and in particular, that its mixing  time is polynomial in $k$ and is given by $\tau_{\mathrm{mix}}(\varepsilon)= \mathcal{O}\left(k^{9/2+\varepsilon}\right)$. Consequentially, the sampling from  $\fmc{b}$ can be performed  in  $\mathcal{O}(\mathcal{P}(k))$ time for any $\varepsilon > 0$. 
 
\smallskip 
  
The proof of Proposition \ref{prop.ssa.fpras} depends on the following technical lemma.

 \begin{lemma}\label{lem.ratio.helper}
   Let $X \sim \widehat{\mydist}$ be a strictly positive univariate random variable such that $a \leq X \leq b$,  and let $X_1,\ldots, X_m$ be its independent realizations. Then,   provided that
    \[
     \| \widehat{\mydist} - \mydist \|_{\mathrm{TV}} \leq \frac{\varepsilon \, a}{4(b-a)}, \quad \text{and} \quad m \geq \frac{(b-a^2)\ln(2/\delta)}{2 (\varepsilon/4)^2 a^2 },
   \]
   it holds that:
  \[
    \Pm\left( (1-\varepsilon) \Em_\mydist[X] \leq \frac{1}{m}\sum_{i=1}^m{X_i}  \leq (1+\varepsilon) \Em_\mydist[X] \right) \geq 1-\delta.
  \]
   \end{lemma}  
 \begin{proof} See Appendix \ref{appendix.proofs}.
  \end{proof} 
  
   We next proceed with the proof of Proposition \ref{prop.ssa.fpras} which is divided into two parts. The first for the conditional expectation estimation and the second for the tail probability evaluation.
  
  \begin{proof}[Proposition \ref{prop.ssa.fpras}: FPRAS for \eqref{eq.sim.prob1}]
  
  With the powerful result of \cite{SinclairCube} in hand, one can achieve a straightforward development of an FPRAS for the conditional expectation estimation problem. The proof follows  immediately from  Lemma \ref{lem.ratio.helper}. In particular, all we need to do in order to achieve an $(\varepsilon,\delta)$ approximation to \eqref{eq.sim.prob1} is to generate
  \[
    m = \frac{(\wmax - \wmin)^2\ln(2/\delta) }{2 (\varepsilon/4)^2 \wmin^2}
  \]
  samples from $\fmc{\treshold}$, such that
  \[
    \left\| \fmc{\treshold} - \mu_\treshold \right\|_{\mathrm{TV}} \leq \frac{\varepsilon \, \wmin}{4(\wmax - \wmin)}.
  \]
   Recall that the mixing time is polynomial in $k$, and note that the number of samples $m$ is also polynomial in $k$, thus the proof is complete, since
  \[
     m = \frac{(\wmax - \wmin)^2\ln(2/\delta) } { (\varepsilon/4)^2\,\wmin^2} \underbrace{=}_{\wmax = \mathcal{O}(\mathcal{P}(k))\wmin}
     %
  %
      \mathcal{O}(\mathcal{P}(k))\frac{\ln(2/\delta)}{\varepsilon^2}.  \eqno \qed 
  \]
  \end{proof}
  
  \begin{proof}[Proposition \ref{prop.ssa.fpras}: FPRAS for \eqref{eq.sim.prob2}]
  In order to put this problem into the setting of
  Theorem~\ref{thm.mcmc} and achieve an
  FPRAS, a careful definition of the corresponding level sets is
  essential. In particular, the number of levels should be polynomial
  in $k$, and the level entrance probabilities $\{ r_t\}_{ 1 \leq t
    \leq n}$, should not be too small. Fix
  \[
    n =  \left \lfloor \frac{\left(\sum_{i=1}^k{w_i}\right) - \treshold}{\wmin} \right \rfloor,
  \]
  to be the number of levels, and set $\gamma_{t} = \treshold + (n-t)\,\wmin$ for $t=0,\ldots,n$. 
  For general $\treshold$ it holds that
  {\footnotesize 
  \begin{align*}
    &\Em\left[\Ind{\{ \perf(\bX) \leq \treshold \}} \right] = \Em\left[\Ind{\{ \perf(\bX) \leq \treshold + n\,\wmin \}} \mid  \perf(\bX) \leq \treshold + n\,\wmin \right] \Pm \left( \perf(\bX) \leq \treshold + n\,\wmin \right) \\
    &+\Em\left[\Ind{\{ \perf(\bX) \leq \treshold \}} \mid   \treshold + n\,\wmin < \perf(\bX) \leq \sum_{i=1}^k{w_i} \right] \Pm \left( \treshold + n\,\wmin < \perf(\bX) \leq \sum_{i=1}^k{w_i} \right) \\
    &=  \underbrace{\Em\left[\Ind{\{ \perf(\bX) \leq \treshold \}} \mid  \perf(\bX) \leq \treshold + n\,\wmin \right]}_{(*)} \frac{2^{k}-1}{2^k} + \left(\sum_{i=1}^k{w_i}\right)\frac{1}{2^k},
  \end{align*}}where the last equality follows from the fact that there is only one vector $\bx=(1,1,\ldots,1)$ for which $\treshold + n\,\wmin < \perf(\bx) \leq \sum_{i=1}^k{w_i}$. That is, it is sufficient to develop an efficient approximation to $(*)$ only, since the rest are constants.
  
   We continue by defining the sets $ {\scX} = {\scX}_{\gamma_0}  \supseteq \cdots  \supseteq {\scX}_{\gamma_n}$ via \eqref{eq.x.exact}, and by noting that
   for this particular problem, our aim is to find $\Pm(\bX \in {\scX}_{\gamma_n})$, so the corresponding estimator simplifies to (see Section \ref{sec.ssa.setup}
  (\ref{rt.explain})),
  \[
    \Ct{} = \prod_{t=0}^n{\Rt{t}}.
  \]
  In order to show that the algorithm provides an FPRAS, we will need to justify only condition (\ref{thm.mcmc.condition1}) of Theorem \ref{thm.mcmc}, which is sufficient in our case because we are dealing with an indicator integrand.  Recall that the formal requirement is
    \[
      \| \fmc{t} - \mu_t \|_{\mathrm{TV}} \leq {\varepsilon \,\underline{r}_t}/{32n}, \quad \text{and} \quad 		|{\cal X}_t| \geq {3072  \,n^2 \ln (4n^2 /\delta)}/{\varepsilon^2 \underline{r}_t^2},
    \]	
  where $\mu_{t}$ is the uniform distribution on  ${\scX}_{\gamma_t}$  for $t=0,\ldots,n$, and each sample in ${\cal X}_t$ is distributed according to $\fmc{t}$. Finally, the FPRAS result is established by noting that the following holds.
  \begin{enumerate}
    \item From Lemma \ref{lem.rt.upper.bound} in the Appendix, we have  that $\underline{r}_t  \geq \frac{1}{k+1}$ for $1 \leq t \leq n$.
    
    \item The sampling from $\fmc{t}$ can be performed in polynomial (in $k$) time  \citep{SinclairCube}.
    
    \item The number of levels $n$ (and thus the required sample size $\{ |{\cal X}_t|\}_{1 \leq t \leq n}$) is polynomial in $k$,  since
  \[
    \left \lfloor \frac{\left(\sum_{i=1}^k{w_i}\right) - \tau}{\wmin} \right \rfloor \leq \frac{k\, \wmax}{\wmin} \underbrace{=}_{\wmax = \mathcal{O}(\mathcal{P}(k))\wmin} \mathcal{O}(\mathcal{P}(k)).  \eqno \qed
  \]
  \end{enumerate}
  \end{proof}
  
  Unfortunately, for many problems, an analytical result such as the one obtained in  this section  is not always possible to achieve. The aim of the following numerical section is to demonstrate that the SSA is capable of handling hard problems in the absence of theoretical performance.

  \section{Numerical experiments}\label{sec.numerics}
  %
  %
  %
  %
  %

  \subsection{Portfolio credit risk}
  We consider a portfolio credit risk setting \citep{Glasserman2}. Given a portfolio of $k$ assets, the portfolio loss $L$ is the random variable
  \beq\label{eq.loss.function}
    L = \sum_{i=1}^k {l_i \; X_i},
  \eeq
  where $l_i$ is the risk of asset $i \in \{1,\ldots,k\}$, and $X_i$ is an indicator  random variable that models the default of asset $i$. Under this setting (and similar to Section \ref{sec.exact.example}), one is generally interested in the following.

  \begin{enumerate}
    \item \emph{Conditional Value at Risk (CVaR).} Given a threshold (value at risk) $v$, calculate the conditional value at risk
  $c = \Em[L \mid L \geq v]$.
  
    \item \emph{Tail probability estimation.} Given the value at risk, calculate the tail probability
    $\Pm(L \geq v) = \Em\left[\Ind{ \{ L \geq v \}} \right]$.
  \end{enumerate}
  
  The SSA can be applied to both problems as follows. For tail probability estimation, we simply set $\Hfunc{}(\bx) = \Ind{\{\sum_{i=1}^k {l_i \; x_i} \geq  v\}}$.  For conditional expectation, we set $\Hfunc{}(\bx) = \sum_{i=1}^k {l_i \; x_i}$. 
  
    Note that the tail probability estimation   (for which the integrand is the indicator function), is a special case of a general integration. Recall that the GS algorithm of \cite{Botev2012} works on indicator integrands, and thus GS is a special case of the SSA. Consequently, in this  section we will investigate the more interesting (and more general) scenario of estimating an expectation conditional on a rare event.
  %
  
  As our working example, we consider a credit risk in a Normal Copula model and, in particular,   a $21$ factor model from \cite{Glasserman2} with $1,000$ obligors. Thus, the integrals of interest are $1021$ dimensional.
  
  The SSA setting is similar to the weighted component model from Section \ref{sec.exact.example}. We  define a $k$-dimensional binary vector $\bx = (x_1,\ldots,x_k)$, for which $x_i$ stands for the $i$th asset default ($x_i=1$ for default, and $0$ otherwise). We take the performance function $S(\bx)$ to be the loss function \eqref{eq.loss.function}. Then, the level sets are defined naturally by $\scX_t = \{ \bx \; : \; S(\bx)\geq \gamma_t\}$,  (see also Section~\ref{sec.ssa.setup}). In our experiment, we set $\gamma_0 = 0$ and $\gamma_n = 1+\sum_{i=1}^k{l_i}$. In order to determine the remaining levels $\gamma_1,\ldots,\gamma_{n-1}$, we execute a pilot run of Algorithm \ref{alg.SSA} with $N=1,000$ and $\rho=0.1$. As an MCMC sampler, we use a Hit-and-Run algorithm \citep[Chapter 10, Algorithm 10.10]{Kroese2011}, taking a new sample after $50$ transitions.
 
   It is important to note that despite the existence of several algorithms for estimating $c$, the SSA has an interesting feature, that (to the best of our knowledge) is not present in other methods. Namely, one is able to obtain an estimator for several CVaRs  via a {\em single} SSA run.
    To see this, consider the estimation of  $c_{1} , \ldots , c_{s}$ for $s\geq 1$. Suppose that $v_{1} \leq \cdots \leq v_{s}$ and note that it will be sufficient to add these values to the $\{ \gamma_t \}$ (as additional levels), and retain $s$ copies of $\Pt{t}$ and $\Ct{t}$. In particular, during the SSA execution, we will need to
   closely follow the $\gamma$ levels, and as soon as we encounter a certain $v_{j}$ for $1 \leq j \leq s$, we will start to update the corresponding values of  $\Pt{t}^{(j)}$ and $\Ct{t}^{(j)}$, in order to allow the corresponding estimation of $c_j = \Em[L \mid L \geq v_j]$.  Despite that such a procedure introduces a dependence between the obtained estimators, they  still remain unbiased.

  To test the above setting, we perform the experiment with a view to estimate  $\{c_j\}_{1 \leq j \leq 13}$ using the following values at risk:
  \begin{align}\label{eq.alpha.levels}
      \{   10000, \;  14000, \;& 18000, \; 22000, \; 24000, \; 28000, \; 30000,   \\
  &  34000, \; 38000, \; 40000, \; 44000, \; 48000, \; 50000 \}. \nonumber
  \end{align}

  The execution of the SSA pilot run with the addition of the desired VaRs (levels) from \eqref{eq.alpha.levels} (marked in bold),  yields  the following level values of $(\gamma_0,\ldots,\gamma_{21})$:
  \begin{align*}
   (& 0, \; 788.3, \; 9616.7, \; \mathbf{10000}, \;	\mathbf{14000}, \;	\mathbf{18000}, \;	\mathbf{22000}, \;	\mathbf{24000}, \;	\mathbf{28000}, \;	\\
   & \mathbf{30000}, \;  \mathbf{34000}, \;	\mathbf{38000}, \;	\mathbf{40000}, \;	\mathbf{44000} , \;	47557.6 , \;	\mathbf{48000} , \;	49347.8 , \;	\\
   & \mathbf{50000}, \;	50320.6 , \,
   50477.4 , \;	50500, \; \infty).
  \end{align*}

  Table \ref{tab.credit.single.ssa}  summarizes the results obtained by executing $1$ pilot and $5$ regular independent runs of the  SSA. For each run, we use the parameter set that was specified for the pilot run ($N=1,000$ and burn-in of $50$). The overall execution time (for all these $R=1+5=6$ independent runs) is 454 seconds. The SSA is very accurate. In particular, we obtain an RE of less than $1\%$ for each $\widehat{c}$ while employing a very modest effort.
  

  \begin{table}[h]  
  \center
  \normalsize
  \begin{tabular}{ccc|ccc}

    $v$ & $\widehat{c}$ & RE &  $v$ & $\widehat{c}$ & RE \\ \hline
  
  10000	&	\num{	1.68E+04	}	&	0.67	\%	&	34000	&	\num{	3.80E+04	}	&	0.05	\%	\\
  14000	&	\num{	2.09E+04	}	&	0.51	\%	&	38000	&	\num{	4.11E+04	}	&	0.05	\%	\\
  18000	&	\num{	2.46E+04	}	&	0.21	\%	&	40000	&	\num{	4.26E+04	}	&	0.08	\%	\\
  22000	&	\num{	2.82E+04	}	&	0.19	\%	&	44000	&	\num{	4.56E+04	}	&	0.08	\%	\\
  24000	&	\num{	2.99E+04	}	&	0.23	\%	&	48000	&	\num{	4.86E+04	}	&	0.02	\%	\\
  28000	&	\num{	3.32E+04	}	&	0.21	\%	&	50000	&	\num{	5.01E+04	}	&	0.02	\%	\\
  30000	&	\num{	3.48E+04	}	&	0.14	\%	&		&				&			\\
    \end{tabular}
  \caption{The SSA results for the Normal copula credit risk model with $21$ factors   and $1,000$ obligors.}\label{tab.credit.single.ssa}
  \end{table}
  
  \normalsize
  
   The obtained result is especially appealing, since the corresponding estimation problem falls into rare-event setting \citep{Glasserman1}.  That is, a CMC estimator will not be applicable in this case.

\enlargethispage{1cm}



\subsection{Self-avoiding walks}\label{sec.saw}
In this section, we consider random walks of length $n$ on the
two-dimensional lattice of integers, starting from the origin. In
particular, we are interested in estimating the following quantities:
\newpage
\begin{enumerate}
	\item $c_n$: the number of SAWs of length $n$,
	\item $\Delta_n$: the expected distance of the final SAW coordinate to the origin.
	
\end{enumerate}

To put these SAW problems into the SSA framework, define the set of
directions, $\scX =
\{\mathrm{Left, Right, Up, Down}\}^n$, and let $f$ be the uniform pdf
on $\scX$. Let $\xi(\bx)$ denote the final coordinate of the random
walk represented by the directions vector $\bx$. 
We have $c_n = \Em_f \left[\Ind{ \{ \bX \text{ is SAW} \}}
  \right]$ and $\Delta_n = \Em_f \left[ \| \xi(\bX) \| \, \, \big|
  \,\Ind{ \{ \bX \text{ is SAW} \}} \right]$.

Next, we let $\scX_t \subseteq \scX$ be the set of all directions
vectors that yield a valid self-avoiding walk of length at least $t$,
for $0 \leq t \leq n$. In addition, we define $\scZ_t$ to be the set
of all directions vectors that yield a self-avoiding walk of 
length (exactly) $t$, for $1 \leq t \leq n$. The
above gives the required partition of~$\scX$. Moreover, the simulation
from $f_t(\bx) = f\left(\bx \mid \bx \in {\scX}_{t-1} \right)$,
reduces to the uniform selection of the SAW's direction at time $1
\leq t \leq n$.

Our experimental setting for SAWs of lengths $n$ is as follows. We set
the sample size of the SSA to be $N_t=1000$ for all $t=1,\ldots,
n$. In this experiment, we are  interested in both the probability
that $\bX$ lies in $\scZ_n$, and the expected distance of $\bX \in
\scZ_n$ (uniformly selected) to the origin. These give us the required  estimators of $c_n$ and $\Delta_n$, respectively.
The leftmost plot of Fig.\,\ref{fig.saw.experimental.error} summarizes a {\em percent error} (PE), which is defined by
\[
	\mathrm{PE} = 100 \, \frac{\widehat{c}_n - c_n}{c_n},
\]
where  $\widehat{c}_n$ stands for  the SSA's estimator of $c_n$.

\begin{figure}[h]	
	\begin{subfigure}{0.5\textwidth}
		 \includegraphics[width=0.8\textwidth]{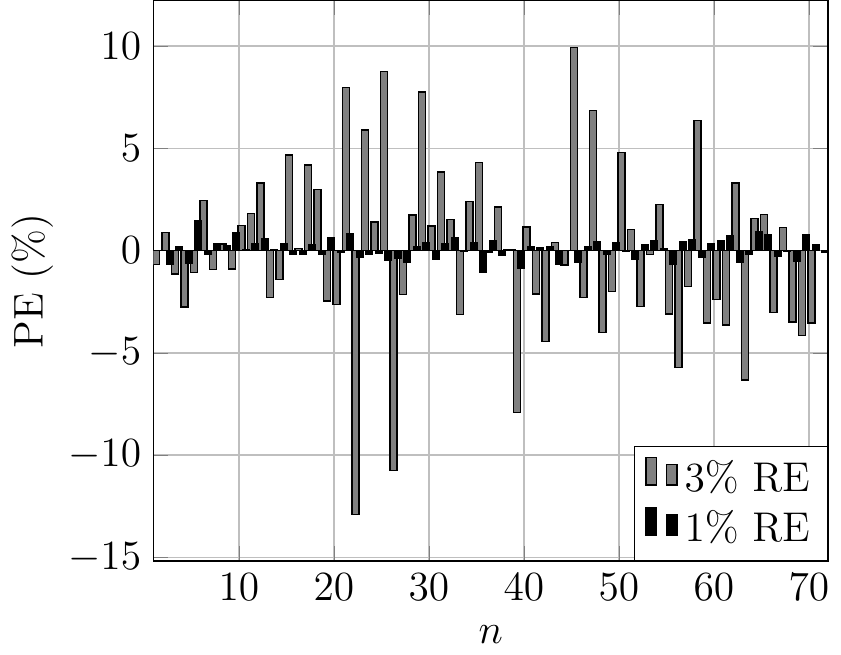}\label{fig.saw.experimental.error1}
	\end{subfigure}
	\quad
	\begin{subfigure}{0.5\textwidth}		
		 \includegraphics[width=0.8\textwidth]{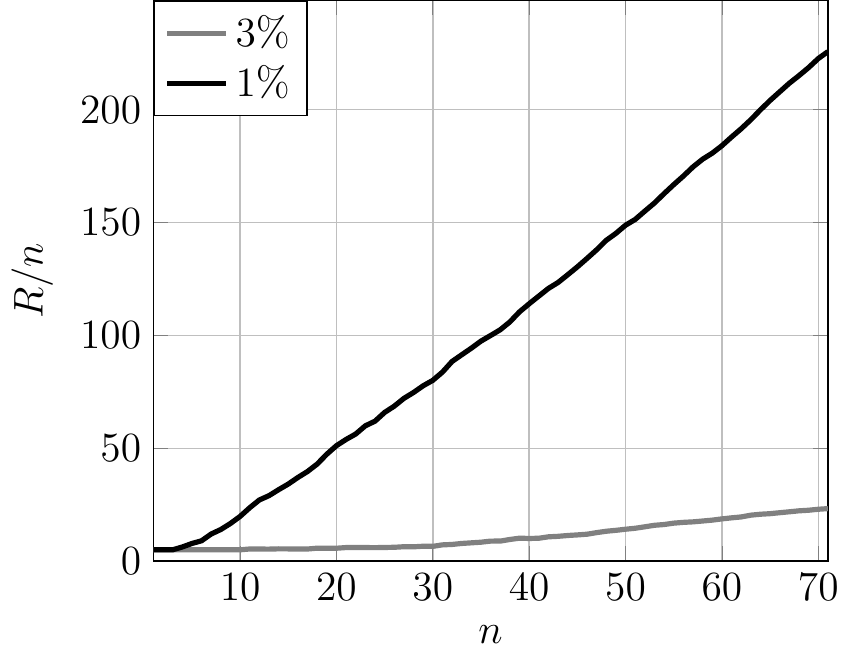}\label{fig.saw.experimental.error2}
	\end{subfigure}	
	\caption{The PE (leftmost plot) and the number of independent runs divided by SAW's length (rightmost  plot) of the SSA  as a function of SAW length $n$ for $3\%$ and $1\%$ RE.} \label{fig.saw.experimental.error}
\end{figure}

In order to explore the convergence of the SSA estimates to the true quantity of interest, the SSA was executed for a sufficient number of times     to obtain  $3\%$ and $1\%$ relative error (RE) \citep{rubinstein17}, respectively. The exact $c_n$ values for $n=1,\ldots,71$ were taken  from \citep{Guttmann2001,Jensen2004}); naturally, when we allow a smaller RE, that is, when we increase $R$, the estimator converges to the true value $c_n$, as can be observed in leftmost plot of Fig.\,\ref{fig.saw.experimental.error}. In addition, the rightmost plot of Fig.\,\ref{fig.saw.experimental.error}  shows that regardless of the RE,  the required number of independent SSA runs $(R)$ divided by SAW's length $(n)$, is  growing linearly with $n$.



  To further investigate the SSA convergence, we consider the two properties of SAW's. In particular, the following holds \citep{PhysRevLett.49.1062,DuminilCopin2013,BEYER1972176,Noonan1998}
  \begin{enumerate}
  	\item $\mu = \lim_{n \to \infty} c_n^{\frac{1}{n}} \in [\underline{\mu}, \overline{\mu} ] = [2.62002, 2.679192495]$.
  	
  	\item 	$\underline{\Delta}_n = n^{\frac{1}{4}} \leq \Delta_n \approx n^{3/4}=\widetilde{\Delta}_n \ll n = \overline{\Delta}_n$.
  \end{enumerate}
  
 Fig.\,\ref{fig.bounds.n} summarizes our results compared to these
 bounds. In particular, we run the SSA to achieve the $3\%$ RE (for
 $\hat{c}_n$) for $1 \leq n \leq 200$. It can be clearly observed,
 that the estimator $\widehat{c}_n^{1/n}$ converges toward the
 $[\underline{\mu}, \overline{\mu} ]$ interval  as $n$ grows. It is
 interesting to note that, at least for small $n$,  $\Delta_n$ seems
 to grow at a rate {\em smaller} than the suggested $3/4$.
  \begin{figure}[H]	
  	\begin{subfigure}{0.5\textwidth}
  		{\includegraphics[width=0.8 \textwidth]{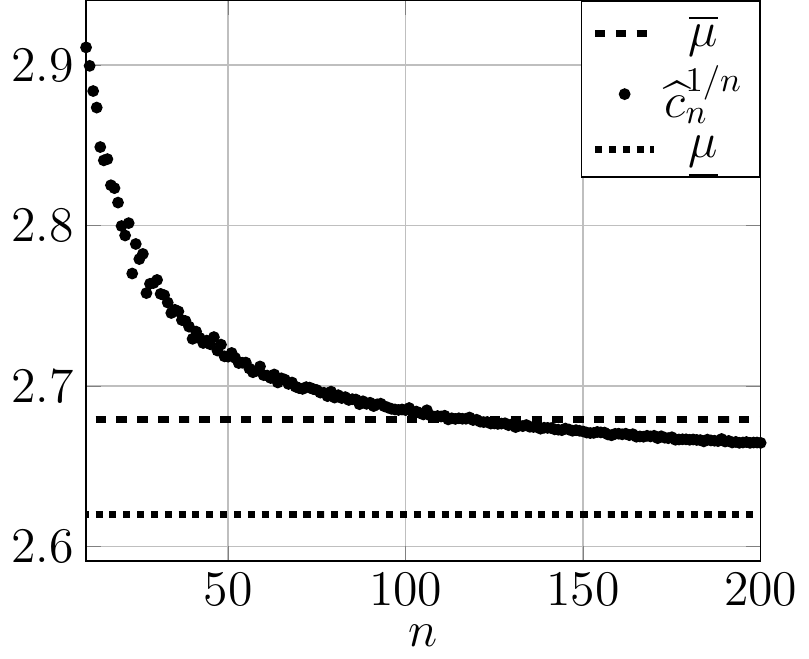}\label{fig.saw.experimental.error11}}	
  	\end{subfigure}
  	\begin{subfigure}{0.5\textwidth}		
  		 {\includegraphics[width=0.8 \textwidth]{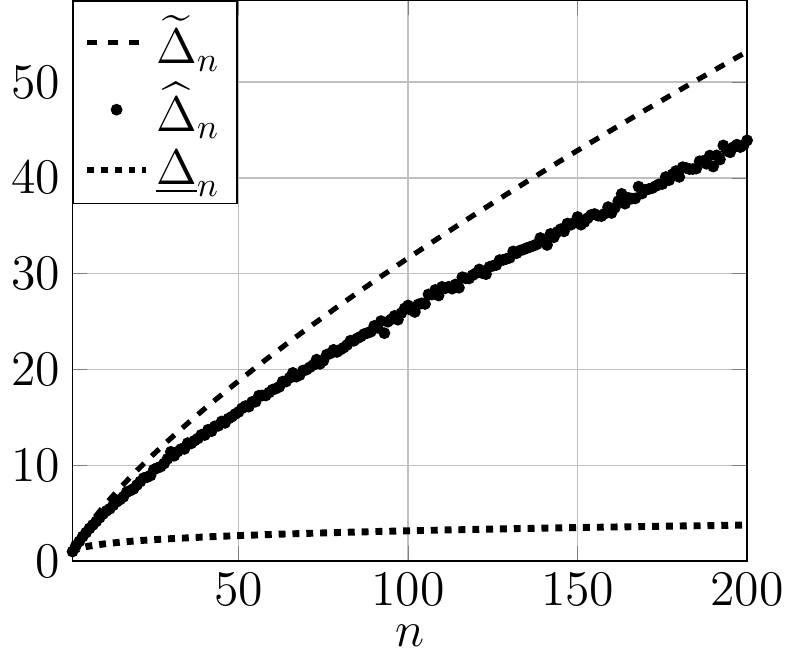}\label{fig.saw.experimental.error21}}
  	\end{subfigure}
  \caption{The $\widehat{c}_n^{1/n}$ and the $\widehat{\Delta}_n$ as a function of the SAW's length $n$.
  } \label{fig.bounds.n}
  \end{figure}
  \section{Discussion}\label{sec.conclusion}
  In this paper we described a general procedure for multi-dimensional
  integration, the SSA, and applied it to various problems from
  different research domains. We showed that this method belongs to a
  very general class of SMC algorithms and developed its theoretical
  foundation. The proposed SSA is relatively easy to implement and our
  numerical study indicates that the SSA yields good performance in
  practice. However, it is important to note that generally speaking,
  the efficiency of the SSA and similar sequential algorithms is
  heavily dependent on the mixing time of the corresponding Markov
  chains that are used for sampling. A rigorous analysis of the mixing
  time for different problems is thus of great interest. Finally,
  based on our numerical study, it will be interesting to apply the
  SSA to other practical problems.

 \section{Acknowledgements} 
 This work was supported by the Australian Research Council Centre of Excellence for Mathematical \& Statistical Frontiers, under grant number CE140100049.
\bibliographystyle{acm}      


\bibliography{SSA}

  \appendix
  \section*{Appendix}
  \small
  \section{Technical arguments}\label{appendix.proofs}
  {\it Proof of Theorem \ref{thm.unbiased}}.
  Recall that  ${\scZ}_1,\ldots, {\scZ}_{n}$ is a partition of the set~${\scX}$, so from the law of total probability we have
  \[
    \Em_f \left[ \Hfunc \left(\bX \right) \right] = \sum_{t=1}^{n} { \Em_f\left[\Hfunc\left(\bX\right) \mid \bX \in  {\scZ}_t \right] \Pm_f\left(\bX \in  {\scZ}_t \right)}.
  \]
  By the linearity of expectation, and since $\Ct{} = \sum_{t=1}^{n}{\Ct{t}}$, it will be sufficient to show that for all $t = 1, \ldots, n$, it holds that
  \[
    \Em \left[ \Ct{t} \right] = \Em_f\left[\Hfunc\left(\bX\right) \mid \bX \in  {\scZ}_t \right] \Pm_f\left(\bX \in  {\scZ}_t \right).
  \]
  To see this, we need the following.
  \begin{enumerate}
    \item Although that the samples in the $\mathcal{Z}_{t}$ set for $1 \leq t \leq n$ are not independent due to MCMC and splitting usage,  they still have the same distribution; that is, for all $t$, it holds  that
    \beq\label{eq.cond.exp}
       \Em \left[ \sum_{\mathbf{X} \in  \mathcal{Z}_t}{\Hfunc(\mathbf{X})} \, \Big| \, \left|\mathcal{Z}_t \right| \right]={\left|\mathcal{Z}_t \right|} \, \Em_f\left[\Hfunc\left(\bX\right) \mid \bX \in  \mathcal{Z}_t \right].
    \eeq    
  \item From the unbiasedness of multilevel splitting \citep{Kroese2011, Botev2012}, it holds for all $1 \leq t \leq n$ that
      \beq\label{eq.ms.unbiasedness}
      \Em \left[ \Pt{t}  \right] =
      \Em\left[ \left(1-\Rt{t}\right)\prod_{j=0}^{t-1}{\Rt{j}}   \right] = 		
       \Pm_f \left( \bX \in {\scZ}_t \right).
    \eeq	
  \end{enumerate}
  Combining \eqref{eq.cond.exp} and \eqref{eq.ms.unbiasedness} with a conditioning on the cardinalities of the $\mathcal{Z}_t$ sets, we complete the proof with:
  
  \begin{align*}
  \Em \left[\Ct{t} \right] &= \Em \left[\Ht{t} \, \Pt{t} \right]
  = \Em \left[\Pt{t} \frac{1}{|\mathcal{Z}_t|}  \sum_{\mathbf{X} \in  \mathcal{Z}_t}{\Hfunc(\mathbf{X})} \right] \\
  & = \Em \left[ \Em \left[
  \Pt{t} \frac{1}{|\mathcal{Z}_t|}  \sum_{\mathbf{X} \in  \mathcal{Z}_t}{\Hfunc(\mathbf{X})}  \, \Bigg| \, |\mathcal{Z}_0|, \ldots, |\mathcal{Z}_{t}|
  \right] \right] \\
  & = \Em \left[ \Pt{t} \frac{1}{|\mathcal{Z}_t|}
  \Em \left[
    \sum_{\mathbf{X} \in  \mathcal{Z}_t}{\Hfunc(\mathbf{X})} \, \Bigg| \,  |\mathcal{Z}_0|, \ldots, |\mathcal{Z}_t|
  \right] \right]\\
   &\underbrace{=}_{\eqref{eq.cond.exp}} \Em \left[  \Pt{t} \frac{1}{|\mathcal{Z}_t|} | \mathcal{Z}_t | \, \Em_f\left[\Hfunc\left(\bX\right) \mid \bX \in  {\scZ}_t \right] \right] \\
  & =   \Em_f\left[H\left(\bX\right) \mid \bX \in  {\scZ}_t \right] \,\Em \left[  \Pt{t}   \right] \underbrace{=}_{\eqref{eq.ms.unbiasedness}}  \Em_f\left[H\left(\bX\right) \mid \bX \in  {\scZ}_t \right] \Pm_f \left( \bX \in {\scZ}_t \right).
  \end{align*}
  \qed

  {\it Proof of Lemma \ref{lem.ratio.helper}}.
  Recall that
  \beq\label{eq.tot.var.bound}
    \| \widehat{\mydist} - \mydist \|_{\mathrm{TV}} = \frac{1}{b-a} \sup_{\Hfunc\,:\,\R \rightarrow [a,b]} {\left| \int{\Hfunc(x)\,\widehat{\mydist}(\di x) } - \int{\Hfunc(x)\,{\mydist}(\di x)} \right|}, \nonumber
  \eeq
  for any function $\Hfunc:\R \rightarrow [a,b]$, (Proposition 3 in \cite{roberts2004general}). Hence,
  \[
     \Em_{{\mydist}}\left[X \right] - (b-a)\frac{\varepsilon \, a}{4(b-a)} \leq \Em_{\widehat{\mydist}}\left[X \right] \leq   \Em_{{\mydist}}\left[X \right]  + (b-a)\frac{\varepsilon \, a}{4(b-a)}.
  \]
  Combining this with the fact that $X \geq a$, we arrive at
  \beq\label{eq.helper1}
   1- \frac{\varepsilon}{4}  \leq    1 - \frac{\varepsilon \, a}{4\Em_{{\mydist}}\left[X \right]} \leq  \frac{\Em_{\widehat{\mydist}}\left[X \right]}{\Em_{{\mydist}}\left[X \right]} \leq 1+ \frac{\varepsilon \, a}{4\Em_{{\mydist}}\left[X \right]}	 \leq 1+ \frac{\varepsilon}{4}.
  \eeq
  Next, since
  \[
    \Em_{\widehat{\mydist}}\left[\frac{1}{m}\sum_{i=1}^m{X_i} \right] = \Em_{\widehat{\mydist}}\left[X \right],
  \]
  we can apply the \cite{hoefding} inequality, to obtain
  \beq\label{eq.helper2}
    \Pm\left( 1 - \frac{\varepsilon}{4} \leq \frac{\frac{1}{m}\sum_{i=1}^m{X_i}}{\Em_{\widehat{\mydist}}\left[X \right]}  \leq 1+ \frac{\varepsilon}{4} \right) \geq 1-\delta,
  \eeq
  for
  \[
    m = \frac{(b-a)^2\ln(2/\delta)}{2 (\varepsilon/4)^2 \left(\Em_{\widehat{\mydist}}\left[X \right]\right)^2} \geq \frac{(b-a)^2\ln(2/\delta)}{2 (\varepsilon/4)^2 a^2}.
  \]
  Finally, we complete the proof by combining \eqref{eq.helper1} and \eqref{eq.helper2}, to obtain:
  \begin{align*}
    &\Pm\left(
    1+\varepsilon \leq \left( 1 - \frac{\varepsilon/2}{2} \right)^2 \leq
    \frac{\Em_{\widehat{\mydist}}\left[X \right]}{\Em_{{\mydist}}\left[X \right]}
    \frac{\frac{1}{m}\sum_{i=1}^m{X_i}}{\Em_{\widehat{\mydist}}\left[X \right]}
    \leq \left( 1 + \frac{\varepsilon/2}{2} \right)^2 \leq 1+\varepsilon
    \right) \\
     &=   \Pm\left( (1-\varepsilon) \Em_\mydist[X] \leq \frac{1}{m}\sum_{i=1}^m{X_i}  \leq (1+\varepsilon) \Em_\mydist[X] \right) \geq 1-\delta. \quad \quad \quad \quad \quad  \qed
  \end{align*}
  \qed

  {\it Proof of Theorem \ref{thm.mcmc}}.
  The proof of this theorem consists of the following steps.
  
  \begin{enumerate}
    \item In Lemma \ref{lem1}, we prove that an existence of an $\left(\varepsilon , \frac{\delta}{n} \right)$-approximation to $\{z_t\}_{1 \leq t \leq n}$ implies an existence of an $\left(\varepsilon , \delta \right)$-approximation to $z=\sum_{t=1}^n{z_t}$.
    
    \item In Lemma \ref{lem2}, we  prove that an existence of an $\left(\frac{\varepsilon}{4} , \frac{\delta}{2n} \right)$-approximation to $\{ \varphi_t \}_{1 \leq t \leq n }$ and $\{ p_t \}_{1 \leq t \leq n }$ implies an $\left(\varepsilon , \frac{\delta}{n} \right)$-approximation existence  to $\{ z_t \}_{1 \leq t \leq n}$.
  
   \item In Lemmas \ref{lem.ratio.helper}, \ref{lem3} and \ref{lem4}, we provide the required $\left(\frac{\varepsilon}{4} , \frac{\delta}{2n} \right)$-approximations to $\varphi_t$ and $p_t$ for $1 \leq t \leq n$.
    
  \end{enumerate}
  
  \begin{lemma}\label{lem1}
  Suppose that for all $t=1,\ldots,n$,  an $\left(\varepsilon , \frac{\delta}{n}\right)$-approximation to $z_t$ exists. Then,
  \[
    \Pm \left(z (1-\varepsilon) \leq \Ct{} \leq z(1+ \varepsilon) \right) \geq 1-\delta.
  \]
  \end{lemma}
  \begin{proof}
  From the assumption of the existence of the $\left(\varepsilon , \frac{\delta}{n} \right)$-approximation to $z_t$ for each $1 \leq t \leq n$, we have
  \[
    \Pm \left( \left|\Ct{t} - z_t\right| \leq \varepsilon z_t \right) \geq 1 - \frac{\delta}{n}, \quad \text{and} \quad \Pm \left( \left|\Ct{t} - z_t\right| > \varepsilon z_t \right) <  \frac{\delta}{n}.
  \]
  By using Boole's inequality (union bound),
    we arrive at
  \[
    \Pm \left(\exists \, t \; : \; \left|\Ct{t} - z_t\right| > \varepsilon z_t \right) \leq \sum_{t=1}^{n}  \Pm \left( \left|\Ct{t} - z_t\right| > \varepsilon z_t \right) < n \frac{\delta}{n} = \delta,
  \]
  that is, it holds for all $t=1,\ldots,n$, that
  \[
    \Pm \left(\forall \, t \; : \; \left|\Ct{t} - z_t\right| \leq \varepsilon z_t \right)
    = 1-\Pm \left(\exists \, t \; : \; \left|\Ct{t} - z_t\right| > \varepsilon c_t \right) \geq 1-\delta,
  \]
  and hence,
  \begin{align*}
    \Pm \left((1-\varepsilon)\sum_{t=1}^{n} {z_t} \leq \sum_{t=1}^{n} {\Ct{t}} \leq (1+\varepsilon)\sum_{t=1}^{n} {z_t}  \right) = \Pm \left( \Ct{} \in \leq z(1 \pm \varepsilon) \right) \geq 1-\delta.
  \end{align*}
  \qed
  \end{proof}

  \begin{lemma}\label{lem2}
  Suppose that for all $t=1,\ldots,n$, there exists an $\left(\frac{\varepsilon}{4} , \frac{\delta}{2n} \right)$-approximation to $\varphi_t$ and $p_t$. Then,
  \[
    \Pm \left(z_t (1-\varepsilon) \leq \Ct{t} \leq z_t(1+ \varepsilon) \right) \geq 1- \frac{\delta}{n} \quad \text{for all }t=1,\ldots,n.
  \]
  \end{lemma}
  
  \begin{proof}
  By  assuming an existence  of $\left(\frac{\varepsilon}{4} , \frac{\delta}{2n} \right)$-approximation to $\varphi_t$ and $p_t$, namely:
  \[
    \Pm \left( \left| \frac{\Ht{t}}{h_t} -1\right| \leq \varepsilon/4 \right) \geq 1-\delta/2n, \quad \text{and} \quad  \Pm \left( \left| \frac{\Pt{t}}{p_t} -1\right| \leq \varepsilon/4 \right)  \geq 1-\delta/2n,
  \]
   and combining it with the union bound, 
   we arrive at
  \[
    \Pm \left( 1-\varepsilon \underbrace{\leq}_{(*)} \left( 1- \frac{\varepsilon/2}{2} \right)^2 \leq \frac{\Ht{t}\Pt{t}}{\varphi_t\,p_t} \leq  \left( 1+ \frac{\varepsilon/2}{2} \right)^2 \underbrace{\leq}_{(*)} 1+ \varepsilon \right) \geq 1-\delta/n,
  \]
  where $(*)$ follows from the fact that for any $0 < |\varepsilon| < 1$ and $n \in \N$ we have
  \beq\label{eq.prop.power.error}
    1-\varepsilon \leq \left( 1 - \frac{\varepsilon/2}{n} \right)^n  \quad  \text{and} \quad \left( 1 + \frac{\varepsilon/2}{n} \right)^n \leq 1+\varepsilon.
  \eeq
  To see that \eqref{eq.prop.power.error} holds, note that by using exponential inequalities from \cite{bullen1998dictionary}, we have that $\left( 1 - \frac{\varepsilon/2}{n} \right)^n \geq 1-\frac{\varepsilon}{2} \geq 1-\varepsilon$. In addition, it holds that $|\e^\varepsilon - 1| < 7\varepsilon/4 $, and hence:
  \[
  \left( 1 + \frac{\varepsilon/2}{n} \right)^n \leq \e^{\varepsilon/2} \leq 1 + \frac{7(\varepsilon/2)}{4} \leq 1+\varepsilon. \eqno \qed
  %
  \] 
  \end{proof}
  
   To complete the proof of Theorem \ref{thm.mcmc}, we need to provide $\left(\frac{\varepsilon}{4} , \frac{\delta}{2n} \right)$-approximations to both $\varphi_t$ and $p_t$.
  However, at this stage we have to take into account a specific splitting strategy,  since the SSA sample size bounds depend on the latter. Here we examine the independent setting, for which the samples in each $\{{\cal X}_t\}_{1 \leq t \leq n}$  set are independent. That is, we use multiple runs of the SSA at each stage $(t=1,\ldots,n)$ of the algorithm execution. See Remark \ref{rem.split.strategy} for further details.

  \begin{remark}[Lemma \ref{lem.ratio.helper} for binary random variables]\label{rem.rt.upper.bound.bin} \normalfont
  For a binary random variable $X \in \{ 0,1\}$, with a known lower bound  on its mean, Lemma \ref{lem.ratio.helper} can be strengthened via the usage of \cite{chernoff1952} bound instead of the \cite{hoefding} inequality. In particular, the following holds.
  
  Let $X \sim \widehat{\mydist}(x)$ be a binary random variable   and let $X_1,\ldots, X_m$ be its independent realizations. Then, provided that
  $\Em_{\widehat{\mydist}}[X] \geq \Em_{\widehat{\mydist}}'[X] $,
    \[
     \| \widehat{\mydist} - \mydist \|_{\mathrm{TV}} \leq \frac{\varepsilon \, \Em_{\widehat{\mydist}}'[X]}{4}, \quad \text{and} \quad m \geq \frac{3\ln(2/\delta)}{(\varepsilon/4)^2 (\Em_{\widehat{\mydist}}'[X])^2 },
   \]
   it holds that:
  \[
    \Pm\left( (1-\varepsilon) \Em_{\mydist}[X] \leq \frac{1}{m}\sum_{i=1}^m{X_i}  \leq (1+\varepsilon) \Em_{\mydist}[X] \right) \geq 1-\delta.
  \]
  The corresponding proof is almost identical to the one presented in Lemma \ref{lem.ratio.helper}. The major difference is the bound on the sample size in \eqref{eq.helper2}, which is achieved via the Chernoff bound from~\cite[Theorem 10.1]{Mitzenmacher} instead of Hoeffding's inequality.
  \end{remark}

  \begin{lemma}\label{lem3}
  Suppose that $a_t = \min_{\bx \in \scZ_t}{\{ \Hfunc(\bx) \}}$, $b_t = \max_{\bx \in \scZ_t}{\{ \Hfunc(\bx) \}}$ for all $t=1,\ldots,n$. Then,
  provided that the samples in the ${\cal Z}_t$ set are independent, and are distributed according to $\pmc{t}$
   such that
    \[
      \| \pmc{t} - \nu_t \|_{\mathrm{TV}} \leq \frac{\varepsilon\, a_t}{16(b_t - a_t)}, \quad \text{and} \quad |{\cal Z}_t| \geq \frac{128(b_t-a_t)^2\ln (4n/\delta)}{\varepsilon^2 a_t^2},
    \]	
  then $\Ht{t} ={|{\cal Z}_t|^{-1}} \sum_{\bX \in {\cal Z}_t}{ \Hfunc(\bX) } $ is an $\left(\frac{\varepsilon}{4} , \frac{\delta}{2n} \right)$-approximation to $\varphi_t$.
   \end{lemma}
  
  \begin{proof}
  The proof is an immediate consequence of Lemma \ref{lem.ratio.helper}. In particular, note that
  \[
    \| \pmc{t} - \nu_t \|_{\mathrm{TV}} \leq \frac{\frac{\varepsilon}{4} \, a_t}{4(b_t-a_t)} = \frac{\varepsilon\, a_t}{16(b_t - a_t)} ,
  \]
  and that
  \[
    |{\cal Z}_t| \geq \frac{(b_t-a_t)^2)\ln(2/\frac{\delta}{2n})}{2 (\frac{\varepsilon}{4}/4)^2 a^2 } = \frac{128(b_t-a_t)^2\ln (4n/\delta)}{\varepsilon^2 a_t^2}. \eqno \qed
  \]
  \end{proof}

  \begin{lemma}\label{lem4}
  Suppose that the samples in the ${\cal X}_t$ set are independent, and are distributed according to $\fmc{t}$, such that
    \[
      \| \fmc{t} - \mu_t \|_{\mathrm{TV}} \leq \frac{\varepsilon \,\underline{r}_t}{32n}, \quad \text{and} \quad 		|{\cal X}_t| \geq \frac{3072  \,n^2 \ln (4n^2 /\delta)}{\varepsilon^2 \underline{r}_t^2},
    \]	
    where $\underline{r}_t = \min{\{r_t,1-r_t\}}$ for $1 \leq t \leq n$. Then, $\Pt{t}$ is an $\left(\frac{\varepsilon}{4} , \frac{\delta}{2n} \right)$-approximation to $p_t$.
  \end{lemma}
  
  \begin{proof}
  Recall that $\Pt{t} = \left(1-\Rt{t}\right)\prod_{j=0}^{t-1}{\Rt{j}}$ for $t=1,\ldots, n$. Again, by combining the union bound  with  \eqref{eq.prop.power.error}, we conclude that the desired approximation to $p_t$ can be obtained by deriving the  $\left(\frac{\varepsilon}{8n} , \frac{\delta}{2n^2}\right)$-approximations for  each $r_t$ and $1-r_t$.  In this case, the probability that for all $t = {1,\ldots,n}$, $\Rt{t}/r_t$ satisfies $1 - \varepsilon/8n \leq  {\Rt{t}}/r_t \leq 1+ \varepsilon/8n$ is at least $1-\delta/2n^2$. The same holds for $({1-\Rt{t}})/(1-r_t)$, and thus, we arrive at:
  \[
    \Pm \left( 1-\varepsilon/4 \leq \left(1 - \frac{\frac{\varepsilon}{4}/2}{n} \right)^n  \leq   \frac{\Pt{t}}{p_t}  \leq  \left(1 + \frac{\frac{\varepsilon}{4}/2}{n} \right)^n \leq 1+\varepsilon/4	 \right) \geq 1-\delta/2n.
  \]
  
  The bounds for each $\Rt{t}$ and $(1-\Rt{t})$ are easily achieved via Remark \ref{rem.rt.upper.bound.bin}. In particular, it is not very hard to verify that in order to get an $\left(\frac{\varepsilon}{8n} , \frac{\delta}{2n^2}\right)$-approximation, it is sufficient to take
  \[
    \| \fmc{t} - \mu_t \|_{\mathrm{TV}} \leq \frac{\frac{\varepsilon}{8n} \,\underline{r}_t}{4} =  \frac{\varepsilon \,\underline{r}_t}{32n},
  \]
  and
  \[
    |{\cal X}_t| \geq \frac{3\ln\left(2/ \frac{\delta}{2n^2}\right)}{(\frac{\varepsilon}{8n}/4)^2} = \frac{3072  \,n^2 \ln (4n^2 /\delta)}{\varepsilon^2 \underline{r}_t^2}. \eqno \qed
  \]
  
  \end{proof}

  \begin{lemma}\label{lem.rt.upper.bound}
  Suppose without loss of generality that $\bw = (w_1,\ldots,w_k)$ satisfies $w_1\leq w_2 \leq  \cdots \leq w_k$, that is $\wmin = w_1$. Then, for
    $\bset$ and $\bwset$ sets defined via \eqref{eq.x.exact}, it holds that:
   \[
     r = \frac{\bwsetcard}{\bsetcard} \geq \frac{1}{k+1}.
   \]
  \end{lemma}
  
  \begin{proof}
  For any $b \in \R$ and $\bx=(x_1,\ldots,x_k)$, define a partition of $\bset$ via
  \[
    \bwsubset = \{\bx \in \scX_b \; : \; x_1 = 1  \}, \quad \text{and} \quad  \bsubset = \{\bx \in \scX_b \; : \; x_1 = 0  \}.
  \]
  Then, the following holds.
  \begin{enumerate}
    \item For any $\bx \in \bwsubset$, replace $x_1 = 1$ with  $x_1 = 0$, and note that the resulting vector is in $\bminussubset$ set, since its performance is at most $b-w_1$, that is $
    \left| \bwsubset \right| \leq \left| \bminussubset \right|$.
    Similarly, for any $\bx \in \bminussubset$, setting $x_1=1$ instead of $x_1=0$,  results in a vector which belongs to the $ \bwsubset$ set. That is:
    \beq\label{k.lemma.helper1}
      \left| \bwsubset \right| = \left|\bminussubset \right|.
    \eeq
    
  \item For any $\bx \in \bwsubset$, replace $x_1 = 1$ with  $x_1 = 0$ and note that the resulting vector is now in the $\bsubset$ set, that is  $|\bwsubset| \leq  |\bsubset|$.  In addition, for any $\bx \in \bsubset$, there are at most $k-1$ possibilities to replace $\bx$'s non-zero entry with zero and set $x_1=1$, such that the result will be in the $\bwsubset$ set. That is, $\left| \bsubset \right| \leq (k-1) \left|\bwsubset \right| +1$, (where $+1$ stands for the vector of zeros), and we arrive at
  \beq\label{k.lemma.helper2}
      \left| \bwsubset \right| \leq 	\left|\bsubset \right| \leq (k-1) 		\left| \bwsubset \right|+1 \leq \ k	\left| \bwsubset \right|.
    \eeq
  \end{enumerate}
  Combining \eqref{k.lemma.helper1} and \eqref{k.lemma.helper2}, we complete the proof by noting that
  \begin{align*}
    \frac{\bwsetcard}{\bsetcard}&= \frac{\left|\bminuswsubset \right| + \left|\bminussubset \right|}{\left|\bwsubset \right| + \left|\bsubset \right|} \geq \frac{ \left|\bminussubset \right|}{ \left|\bwsubset \right| + \left|\bsubset \right|}\\	
  & \underbrace{=}_{\eqref{k.lemma.helper1},\eqref{k.lemma.helper2}} \frac{ \left|\bwsubset \right|}{ \left|\bwsubset \right| + k \left|\bwsubset \right|} = \frac{1}{k+1}.  
  \end{align*}
  \qed
  \end{proof}

\end{document}